\documentclass[12pt]{amsart}
\usepackage{amsfonts, amsbsy, amsmath, amssymb}

\usepackage[T1]{fontenc}
\usepackage[utf8]{inputenc}
\usepackage{lmodern}
\usepackage{caption}
\usepackage{amsmath, amsthm, amssymb,amscd, mathrsfs, amsfonts, mathtools,tikz-cd}
\usepackage{relsize}
\usepackage{euler}
\usepackage{pbox}
\usepackage{times}
\usepackage[all]{xy}
\usepackage{todonotes}
\usepackage{xcolor}
\usepackage[bookmarks=true,
bookmarksnumbered=true, breaklinks=true,
pdfstartview=FitH, hyperfigures=false,
plainpages=false, naturalnames=true,
colorlinks=true,pagebackref=true,
pdfpagelabels]{hyperref}

\usepackage{tikz}
\usepackage{pgfplots}

\pgfplotsset{compat=1.10}
\usepgfplotslibrary{fillbetween}
\usetikzlibrary{patterns}
\usetikzlibrary{matrix}

\usepackage[]{amsrefs}

\renewcommand{\PrintDOI}[1]{\href{http://dx.doi.org/\detokenize{#1}}{doi: \detokenize{#1}}%
	\IfEmptyBibField{pages}{, (to appear in print)}{}}

\def\commutatif{\ar@{}[rd]|{\circlearrowleft}}

\newtheorem{thm}{Theorem}[section]
\newtheorem{prop}[thm]{Proposition}
\newtheorem{lem}[thm]{Lemma}
\newtheorem{cor}[thm]{Corollary}

\theoremstyle{definition}
\newtheorem{defn}[thm]{Definition}

\theoremstyle{remark}
\newtheorem{rmk}[thm]{Remark}

\newtheorem{ex}[thm]{Example}

\allowdisplaybreaks

\usepackage{bbm}

\hypersetup{
	colorlinks = true,
	urlcolor = blue,
	linkcolor = blue,
	citecolor = red,
	pdfpagemode = UseNone
}

\title{Notes from a family of smooth $G$-Hilbert schemes}

\author{Boris Tsvelikhovskiy} 
\address{Department of Mathematics, University of California, Riverside, CA 92521, USA} 
\email{borist@ucr.edu}
\begin{document}
\maketitle 

\pgfkeys{/pgfplots/scale/.style={
  x post scale=#1,
  y post scale=#1,
  z post scale=#1}
}
\pgfkeys{/pgfplots/axis labels at tip/.style={
    xlabel style={at={(current axis.right of origin)}, xshift=1.5ex, anchor=center},
    ylabel style={at={(current axis.above origin)}, yshift=1.5ex, anchor=center}}
}

\begin{abstract}
Let $ G = \mathbb{Z}/r\mathbb{Z}$  be the cyclic group of order $r$, and let $\varpi = e^{2\pi i / r}$  denote a primitive $r$ th root of unity. 
Consider the action of $G$ on $\mathbb{C}^n$ via the embedding
$$
    \varphi : G \hookrightarrow GL_n(\mathbb{C}), \qquad 
    \varphi(1) = \mathrm{diag}\!\bigl(
        \underbrace{\varpi, \dots, \varpi}_{s},\,
        \underbrace{\varpi^{-1}, \dots, \varpi^{-1}}_{n - s}
    \bigr),
$$
where $0 < s < n $. Denote the corresponding GIT quotient by
$$
\mathcal{X}_{s,n,r} = \mathrm{Spec}\bigl((\mathbb{C}[z_1,\dots,z_n])^G\bigr).
$$
Then the varieties $\mathcal{X}_{s,n,r}$ is a cyclic quotient singularity of type $\tfrac{1}{r}\bigl(\underbrace{1,\dots,1}_{s}, \underbrace{-1,\dots,-1}_{n-s}\bigr)$.

We show that the associated \(G\)-Hilbert schemes \(\mathcal{Y}_{s,n,r}\) are smooth, connected, and irreducible. The natural morphism
\[
\rho_{s,n,r}:\mathcal{Y}_{s,n,r}\longrightarrow\mathcal{X}_{s,n,r}
\]
is a projective resolution of \(\mathcal{X}_{s,n,r}\), discrepant for \(n \ge 3\). We establish that the irreducible components of the central fiber \(\rho_{s,n,r}^{-1}(0)\) are in bijection with the nontrivial characters of \(G\), thereby realizing the classical McKay correspondence in this family of examples.

Finally, we describe a canonical choice of this bijection via the Fourier--Mukai type functor
\[
\Psi : D^b(\mathrm{Coh}_G(\mathbb{C}^n)) \longrightarrow D^b(\mathrm{Coh}(\mathcal{Y}_{s,n,r})),
\]
by showing that, for each nontrivial irreducible representation of $G$, the corresponding skyscraper sheaf is mapped to a complex whose $0^{\text{th}}$ cohomology is supported on a unique irreducible component of the central fiber $\rho_{s,n,r}^{-1}(0)$.
\end{abstract}

\section{Introduction}

The classical McKay correspondence begins with a finite subgroup $G \subset SL_2(\mathbb{C})$ and the categorical quotient $\mathbb{C}^2 // G$. This variety has an isolated singularity at the origin, commonly known as a \emph{Kleinian singularity}. A natural minimal resolution of this singularity is provided by the $G$-Hilbert scheme, which is a fine moduli space parameterizing zero-dimensional subschemes $\mathcal{Z} \subset \mathbb{C}^2$ such that $H^0(\mathcal{O}_\mathcal{Z})$ is isomorphic to the regular representation of $G$ (see \cite{IN2}). The corresponding resolution map is the Hilbert--Chow morphism.  

More generally, it is well known that $G$-Hilbert schemes are smooth when $G$ is a finite subgroup of $SL_2(\mathbb{C})$ or $SL_3(\mathbb{C})$, and in these cases the morphism 
\[
   G\operatorname{-}\mathrm{Hilb}(\mathbb{C}^n) \longrightarrow \mathbb{C}^n // G
\]
provides a minimal resolution of singularities (see \cite{BKR} and related works). For $n > 3$, however, $G$-Hilbert schemes are rarely smooth.  

In this paper, we focus on a particular family of embeddings of the cyclic group  \( G = \mathbb{Z}/r\mathbb{Z} \) into \( GL_n(\mathbb{C}) \), given by
\[
     1 \longmapsto \mathrm{diag}(\underbrace{\varpi, \ldots, \varpi}_s, \underbrace{\varpi^{-1}, \ldots, \varpi^{-1}}_{n-s}),
   \qquad \varpi = e^{2\pi i / r}, \quad 0 < s < n.
\]

We denote the resulting categorical quotient by
\[
   \mathcal{X}_{s,n,r} := \operatorname{Spec}(\mathbb{C}[z_1, \ldots, z_n]^G),
\]
which has a singularity of type $\tfrac{1}{r}(\underbrace{1, \ldots, 1}_s, \underbrace{-1, \ldots, -1}_{n-s})$. The associated $G$-Hilbert scheme will be denoted by $\mathcal{Y}_{s,n,r}$.  

Our first main result asserts that all schemes $\mathcal{Y}_{s,n,r}$ are smooth, connected, and irreducible (Theorem~\ref{Smooth} and Lemma~\ref{lem:connected_irred}). 
The crucial observation underlying this result is that both $\mathcal{X}_{s,n,r}$ and $\mathcal{Y}_{s,n,r}$ can be realized as toric varieties, allowing us to exploit standard techniques from toric geometry.

The (classical) McKay correspondence establishes a bijection between the nontrivial irreducible representations of $G$ and the irreducible components of the exceptional locus, $\rho^{-1}(0)$, of a resolution $\rho : Y \to X$. 
This correspondence is known to hold for finite subgroups of $SL_2(\mathbb{C})$ and $SL_3(\mathbb{C})$ when $Y$ is realized as the associated $G$-Hilbert scheme (see \cite{McKay, BKR} for the foundational results and \cite{IN3} for the case of finite subgroups $G \subset GL_2(\mathbb{C})$).

We prove that for any resolution 
\(\rho : \mathcal{Y}_{s,n,r} \to \mathcal{X}_{s,n,r}\), 
the central fiber \(\rho^{-1}(0)\) decomposes as a union of \(r-1\) exceptional divisors, thus establishing that the McKay correspondence holds in our setting (Corollary~\ref{DivisorLem}). 
We remark, however, that for \(n \ge 3\) these resolutions are discrepant.

Furthermore, we provide a natural way to choose the bijection, in close analogy with the classical case $n=2$ studied by Kapranov and Vasserot in \cite{KV}. In particular, we show that a Fourier–Mukai type functor
\[
   \Psi : D^b(\mathrm{Coh}_G(\mathbb{C}^n)) \longrightarrow D^b(\mathrm{Coh}(\mathcal{Y}_{s,n,r})),
\]
appearing in \cite{CL}, maps the skyscraper sheaf corresponding to a nontrivial irreducible representation of $G$ to a complex whose $0^\text{th}$ cohomology is supported on a unique irreducible component of $\rho^{-1}(0)$ (see Proposition \ref{prop:top-bottom-cohomology}).

The exposition of the paper is organized as follows. In Section~\ref{sec2}, we introduce the setup, fix the notation and describe the categorical quotients under consideration, realizing them as toric varieties. Section~\ref{sec3} contains the main results on the geometry of the $G$-Hilbert schemes $\mathcal{Y}_{s,n,r}$. We show that these schemes are smooth (Theorem~\ref{smooththm}), describe them explicitly as toric varieties (Theorem~\ref{ToricHilbThm}), and use this description to identify the corresponding exceptional divisors (Corollary~\ref{DivisorLem}). Furthermore, we compute the discrepancy divisor in Lemma~\ref{DiscrepDivLemma} 
and deduce that the resolutions 
\(
\rho\colon \mathcal{Y}_{s,n,r} \longrightarrow \mathcal{X}_{s,n,r}
\)
are discrepant whenever $n>2$ (see Corollary~\ref{discrepDivisor}).
 Section~\ref{sec4} is devoted to applications of these results in the context of the McKay correspondence. Finally, in Section~\ref{sec5}, we realize $\mathcal{Y}_{s,n,r}$ as a moduli space of representations of a suitable quiver, following the framework developed in~\cite{IN}. This description allows us to prove that $\mathcal{Y}_{s,n,r}$ is connected and irreducible (Lemma~\ref{lem:connected_irred}).

\textbf{Acknowledgements.} I would like to thank Yujiro Kawamata, Andreas Krug and Timothy Logvinenko for stimulating discussions on the subject.

\section{Definitions and first results}\label{sec2}
Throughout this paper, we consider the following setup. Let 
\(
G = \mathbb{Z}/r\mathbb{Z}
\) 
be the cyclic group of order $r$, and fix an embedding 
\begin{equation}\label{eq:G-action}
    \varphi: G \hookrightarrow GL_n(\mathbb{C}), \qquad \varphi(1) = \mathrm{diag}(\varpi^{a_1}, \ldots, \varpi^{a_n}),
\end{equation}

where $\varpi = e^{2\pi i / r}$ is a primitive $r$th root of unity. We assume that each exponent \(a_i\) is either \(1\) or \(r-1\). Let 
\[
s := \#\{\, i \mid a_i = 1 \,\}
\] 
denote the number of exponents equal to \(1\), and assume \(0 < s < n\), so that both values occur among the \(a_i\). Up to renumbering of the coordinates, we may take
\[
(a_1, \ldots, a_n) = (\underbrace{1, \ldots, 1}_{s\text{ times}}, \underbrace{r-1, \ldots, r-1}_{n-s\text{ times}}).
\]

Recall that a \emph{cluster} \( \mathcal{Z} \subset \mathbb{C}^n \) is a zero-dimensional subscheme, i.e. a finite-length subscheme supported on finitely many points of \( \mathbb{C}^n \).  
A \emph{$G$-cluster} is a $G$-invariant cluster satisfying
\[
H^0(\mathcal{O}_{\mathcal{Z}}) \simeq \mathrm{Reg},
\]
where $H^0(\mathcal{O}_{\mathcal{Z}})$ denotes the space of global functions on $\mathcal{Z}$, and $\mathrm{Reg}$ is the regular representation of $G$.  In particular, \( \dim H^0(\mathcal{O}_{\mathcal{Z}}) = |G| \), and each irreducible representation of \( G \) appears with multiplicity equal to its dimension.

The \emph{$G$-Hilbert scheme}, written $G\operatorname{-}\mathrm{Hilb}(\mathbb{C}^n)$, is the fine moduli space parameterizing all such $G$-clusters.

For the cyclic group \( G = \mathbb{Z}/r\mathbb{Z} \) considered in this paper, the regular representation decomposes into a direct sum of its one-dimensional irreducible representations:
\[
\mathrm{Reg} = \bigoplus_{k=0}^{r-1} \chi_k,
\]
where each \( \chi_k:G\to\mathbb{C}^{*} \) denotes a distinct one-dimensional character of \( G \).

Let \(\mathcal{R} := \mathbb{C}[z_1, \ldots, z_n]\), and consider the categorical quotient
\begin{equation}\label{eq:DefOfX}
    \mathcal{X}_{s,n,r} := \mathrm{Spec}\!\big(\mathcal{R}^G\big)
\end{equation}

of \(\mathrm{Spec}(\mathcal{R}) = \mathbb{C}^n\) by the natural action of \(G\). We denote by \(\mathcal{Y}_{s,n,r}\) the corresponding \(G\)-Hilbert scheme.

The natural morphism
\begin{equation}\label{eq:Resolution}
\rho : \mathcal{Y}_{s,n,r} \longrightarrow \mathcal{X}_{s,n,r},
\quad
I \longmapsto \mathrm{supp}(I),
\end{equation}
is called the \emph{Hilbert--Chow morphism}. By construction, it is birational, mapping each $G$-cluster to its support in the quotient variety.

\begin{rmk}
Consider the diagonal matrix
\[
\mathrm{diag}\!\bigl(\underbrace{\varpi,\ldots,\varpi}_{s},\, 
\underbrace{\varpi^{\,r-1},\ldots,\varpi^{\,r-1}}_{n-s}\bigr),
\]
where \(\varpi\) is a primitive \(r\)th root of unity. Since \(\varpi^{\,r}=1\), we have
\[
\Bigl(\mathrm{diag}\!\bigl(\underbrace{\varpi,\ldots,\varpi}_{s},\, 
\underbrace{\varpi^{\,r-1},\ldots,\varpi^{\,r-1}}_{n-s}\bigr)\Bigr)^{r-1}
=
\mathrm{diag}\!\bigl(\underbrace{\varpi^{\,r-1},\ldots,\varpi^{\,r-1}}_{s},\, 
\underbrace{\varpi,\ldots,\varpi}_{n-s}\bigr).
\]

Consequently, the associated quotient varieties and Hilbert schemes  are identified:
\[
\mathcal{X}_{s,n,r} = \mathcal{X}_{n-s,n,r}, 
\qquad
\mathcal{Y}_{s,n,r} = \mathcal{Y}_{n-s,n,r}.
\]
In other words, exchanging the roles of \(\varpi\) and \(\varpi^{\,r-1}\) produces the same varieties.
\label{sameVarieties}
\end{rmk}

\subsection{Singularity Type of $\mathcal{X}_{s,n,r}$}

Since the action of \( G \) on \( \mathbb{C}^n \) given in \eqref{eq:G-action} is free away from the origin, 
each variety \( \mathcal{X}_{s,n,r} \) has an isolated singular point 
\[
o \in \operatorname{Spec}(\mathcal{R}^G).
\]
Recall the standard terminology:
a singularity is called \emph{Gorenstein} if the canonical divisor is Cartier,
\emph{canonical} if all discrepancies of exceptional divisors are nonnegative,
and \emph{terminal} if all discrepancies are strictly positive.

We determine the type of the singularity at \( o \)
using the numerical criteria summarized in Theorem~2.3 of~\cite{MS}
(see also the references cited therein immediately preceding the theorem).

\begin{thm}
Let $n>2$. The singularity at $o \in \mathcal{X}_{s,n,r}$ is
\begin{enumerate}
    \item canonical;
    \item terminal;
    \item Gorenstein if and only if $2s-n \equiv 0 \pmod{r}$.
\end{enumerate}
\label{singType}
\end{thm}

\begin{proof}
By Theorem~2.3 of \cite{MS}, the singularity is canonical if 
\[
\frac{1}{r}\big(sk+(n-s)(r-k)\big) 
= n-s+\frac{(2s-n)k}{r} \;\;\geq 1
\quad \text{for all } 1\leq k \leq n-1,
\]
and terminal if the inequalities are strict. 

Since $\mathcal{X}_{s,n,r} \cong \mathcal{X}_{n-s,n,r}$ (see Remark~\ref{sameVarieties}), we may assume without loss of generality that 
\(\tfrac{n}{2} \le s < n\).  
In this range, the expression above satisfies
\[
n-s + \frac{(2s-n)k}{r} > 1,
\]
which establishes that the singularity is terminal.  

Finally, the singularity is Gorenstein if and only if the image of $G$ is contained in $SL_n(\mathbb{C})$, that is, when the determinant of the action is trivial. 
In our setting, this condition amounts to 
\(
\varpi^{s-(n-s)} = \varpi^{2s-n} = 1,
\) 
which is equivalent to the congruence 
\[
2s - n \equiv 0 \pmod{r}.
\]

\end{proof}

\subsection{Toric description of $\mathcal{X}_{s,n,r}$}
Define the lattice 
\begin{equation}\label{eq:LatticeN}
    N = \mathbb{Z}^n + \tfrac{1}{r}(\underbrace{1, \ldots, 1}_{s}, \underbrace{r-1, \ldots, r-1}_{n-s})\mathbb{Z}, 
\end{equation}
 and let \(M = \mathrm{Hom}_{\mathbb{Z}}(N,\mathbb{Z}),
\) be the dual lattice. Consider the strongly convex rational polyhedral cone
\begin{equation}\label{eq: ConeSigma}
    \sigma = \mathbb{R}_{\geq 0} e_1 + \cdots + \mathbb{R}_{\geq 0} e_n 
\;\subset\; N \otimes_{\mathbb{Z}} \mathbb{R},
\end{equation}
generated by the standard basis vectors.  

Then the invariant ring of $\mathcal{R}$ under $G$ can be described as
\[
\mathcal{R}^G \;\simeq\; \mathbb{C}[\sigma^{\vee} \cap M],
\]
where $\sigma^{\vee}:=\{\, u \in M \mid \langle u, v \rangle \ge 0 \;\; \forall\, v \in \sigma \,\} \subset M \otimes_{\mathbb{Z}} \mathbb{R}$ denotes the dual cone of $\sigma$.  

Equivalently, the categorical quotient
\[
\mathcal{X}_{s,n,r}=\mathbb{C}^n // G \;\simeq\; \mathrm{Spec}(\mathcal{R}^G)
\]
is the affine toric variety associated with the cone $\sigma$ in the lattice $N$.


\begin{defn}
Let \( N_0 = \mathbb{Z}^n \) and \( M_0 = \mathrm{Hom}_{\mathbb{Z}}(N_0, \mathbb{Z}) \).  
Define the group homomorphism
\[
\mathrm{wt} : M_0 \longrightarrow G
\]
which assigns to each lattice element its \(G\)-grading under the diagonal action of \(G\) on monomials.
\end{defn}

Consider the basis of \(M_0\) given by the exponents of the coordinate functions \(z_1, \dots, z_n \in \mathcal{R}\).  
Then the \(G\)-weights for the action described in \eqref{eq:G-action} are
\begin{equation}\label{eq: G-weights}
\mathrm{wt}(z_i) = r-1, \quad 1 \le i \le s, \qquad
\mathrm{wt}(z_j) = 1, \quad s+1 \le j \le n.
\end{equation}
Here, we slightly abuse notation and write \(z_i\) for the corresponding standard basis vector \(e_i \in M_0\).

\section{Geometric properties of $\mathcal{Y}_{s,n,r}$}\label{sec3}

The goal of this section is to better understand the structure of $G$-Hilbert schemes $\mathcal{Y}_{s,n,r}$.

\subsection{Torus fixed points, smoothness and description of central fiber}

Let $T \subset GL_n(\mathbb{C})$ be the diagonal torus in the basis corresponding to the coordinate functions \( z_1, \ldots, z_n \). The induced action of $T$ on $\mathcal{Y}_{s,n,r}$ provides a powerful framework for analyzing its geometric structure. In particular, the fixed points of this torus action correspond to certain monomial ideals, and examining their Zariski tangent spaces allows to establish the smoothness.

\begin{prop}\label{Smooth}
\begin{enumerate}
    \item The scheme $\mathcal{Y}_{s,n,r}$ has $s(n-s)(r-2)+n$ torus $T$-fixed points.
    \item The dimension of the Zariski tangent space at each fixed point $p$ is equal to $n$, i.e.,
    \[
    \dim \mathrm{Hom}_{\mathcal{R}}(I_p, \mathcal{R}/I_p)^G = n,
    \]
    where $I_p\subset\mathcal{R}$ is the ideal corresponding to $p$.
\end{enumerate}
\end{prop}
\begin{ex}
Let $G = \mathbb{Z}/5\mathbb{Z}$ and $\varphi: G \hookrightarrow GL_3(\mathbb{C})$ be the embedding given by 
\[
\varphi(1) = \mathrm{diag}(\varpi, \varpi, \varpi^4).
\] 
Then the formula $2(3-2)(5-2)+3=9$ predicts $9$ fixed points for the torus $T$-action. Each fixed point $p$ corresponds to an ideal $I_p$ such that $\mathbb{C}[x,y,z]/I_j$ is $5$-dimensional and is a regular representation of $G$. For any fixed point $p$ corresponding to ideal $I_p$, the basis of $\mathbb{C}[x,y,z]/I_p\simeq \mathbb{C}^5$ (regular representation of $\mathbb{Z}/5\mathbb{Z}$) consists of monomials. 

As $\mathrm{wt}(x) = \mathrm{wt}(y) = r-1$, either $x$ or $y$ must lie in $I_p$. If $y \in I_p$, then $\mathrm{wt}(xz) = \mathrm{wt}(1) = 0$ implies $xz \in I_p$ as well, leading to the following possibile pairs of ideals and respective quotients:
\[
\begin{aligned}
(1)\;& I_{\Gamma_1} = \langle x^5,\, y,\, z \rangle, 
&\quad& \Gamma_1 = \{1,\, x,\, x^2,\, x^3,\, x^4\}; \\[3pt]
(2)\;& I_{\Gamma_2} = \langle x^4,\, y,\, z^2,\, xy \rangle, 
&\quad& \Gamma_2 = \{1,\, x,\, x^2,\, x^3,\, z\}; \\[3pt]
(3)\;& I_{\Gamma_3} = \langle x^3,\, y,\, z^3,\, xy \rangle, 
&\quad& \Gamma_3 = \{1,\, x,\, x^2,\, z,\, z^2\}; \\[3pt]
(4)\;& I_{\Gamma_4} = \langle x^2,\, y,\, z^4,\, xy \rangle, 
&\quad& \Gamma_4 = \{1,\,x,\, z,\, z^2,\, z^3,\}; \\[3pt]
(5)\;& I_{\Gamma_5} = \langle x,\, y,\, z^5,\, xy \rangle, 
&\quad& \Gamma_5 = \{1,\, z,\, z^2,\, z^3,\, z^4\}.
\end{aligned}
\]

Analogously, if $x \in I_p$, we obtain:
\[
\begin{aligned}
(6)\;& I_{\Gamma_6} = \langle x,\, y^5,\, z \rangle, 
&\quad& \Gamma_6 = \{1,\, y,\, y^2,\, y^3,\, y^4\}; \\[3pt]
(7)\;& I_{\Gamma_7} = \langle x,\, y^4,\, z^2,\, yz \rangle, 
&\quad& \Gamma_7 = \{1,\, y,\, y^2,\, y^3,\, z\}; \\[3pt]
(8)\;& I_{\Gamma_8} = \langle x,\, y^3,\, z^3,\, yz \rangle, 
&\quad& \Gamma_8 = \{1,\, y,\, y^2,\, z,\, z^2\}; \\[3pt]
(9)\;& I_{\Gamma_9} = \langle x,\, y^2,\, z^4,\, yz \rangle, 
&\quad& \Gamma_9 = \{1,\, y\, z,\, z^2,\, z^3,\}.
\end{aligned}
\]

\end{ex}
Now we present the proof of Proposition \ref{Smooth}.

 \begin{proof}[Proof of Proposition \ref{Smooth}]
Let \(y_1, \dots, y_s\) denote the coordinate functions of \(G\)-weight \(r-1\), and let \(x_1, \dots, x_{n-s}\) denote those of weight \(1\), as given in \eqref{eq: G-weights}.

Since the ideal \(I_p \subset \mathcal{R}\) is fixed by the torus \(T\), it must be a monomial ideal.  Moreover, as $\mathcal{R}/I_p \simeq \bigoplus\limits_{k=0}^{r-1} \chi_k$, all but at most one of the \(y_i\)'s belong to \(I_p\), and the same holds for the \(x_j\)'s. Suppose \(x_j, y_i \notin I_p\) for some \(i,j\). Then, since \(\mathrm{wt}(y_i x_j) = 0\), we must have \(y_i x_j \in I_p\). This leads to fixed points corresponding to ideals (and quotients) of the form:

\[
\begin{aligned}
& I_2^{\,ij}  = \langle\, y_1,\, \dots,\, y_i^2,\, \dots,\, y_s,\; x_1,\, \dots,\, x_j^{r-1},\, \dots,\, x_{n-s},\; y_i x_j \,\rangle,
&& \Gamma_2^{\,ij}  = \{1,\, x_j,\, \dots,\, x_j^{r-2},\,y_i\}; \\[6pt]
& \vdots && \vdots \\[6pt]
& I_{r-1}^{\,ij} = \langle\, y_1,\, \dots,\, y_i^{r-1},\, \dots,\, y_s,\; x_1,\, \dots,\, x_j^2,\, \dots,\, x_{n-s},\; y_i x_j \,\rangle,
&& \Gamma_{r-1}^{\,ij} = \{1,\, x_j,\, y_i^{r-2},\, \dots,\, y_i\}.
\end{aligned}
\]

If all \(y_1, \dots, y_s \in I_p\) or all \(x_1, \dots, x_{n-s} \in I_p\), we obtain:

\[
\begin{aligned}
& I_1^{\,i}   = \langle\, y_1,\, \dots,\, y_i,\, \dots,\, y_s,\; x_1,\, \dots,\, x_j^r,\, \dots,\, x_{n-s},\; y_i x_j \,\rangle,
&& \Gamma_1^{\,i}   =\{1,\, x_j,\, x_j^2,\, \dots,\, x_j^{r-1}\}; \\[4pt]
& I_r^{\,j}   = \langle\, y_1,\, \dots,\, y_i^r,\, \dots,\, y_s,\; x_1,\, \dots,\, x_j,\, \dots,\, x_{n-s},\; y_i x_j \,\rangle,
&& \Gamma_r^{\,j}   = \{1,\, y_i^{r-1},\, y_i^{r-2},\, \dots,\, y_i\}; \\[4pt]
\end{aligned}
\]

Combining these cases, the complete list of \(T\)-fixed ideals is:

\begin{equation}
\label{eq:fixedPoints}
\begin{aligned}
& I_2^{\,ij}  = \langle\, y_1,\, \dots,\, y_i^2,\, \dots,\, y_s,\; x_1,\, \dots,\, x_j^{r-1},\, \dots,\, x_{n-s},\; y_i x_j \,\rangle,
&& \Gamma_2^{\,ij}  = \{1,\, x_j,\, \dots,\, x_j^{r-2},\,y_i\}; \\[6pt]
& \vdots && \vdots \\[6pt]
& I_{r-1}^{\,ij} = \langle\, y_1,\, \dots,\, y_i^{r-1},\, \dots,\, y_s,\; x_1,\, \dots,\, x_j^2,\, \dots,\, x_{n-s},\; y_i x_j \,\rangle,
&& \Gamma_{r-1}^{\,ij} = \{1,\, x_j,\, y_i^{r-2},\, \dots,\, y_i\},\\[4pt]
& I_1^{\,i}   = \langle\, y_1,\, \dots,\, y_i,\, \dots,\, y_s,\; x_1,\, \dots,\, x_j^r,\, \dots,\, x_{n-s},\; y_i x_j \,\rangle,
&& \Gamma_1^{\,i}   =\{1,\, x_j,\, x_j^2,\, \dots,\, x_j^{r-1}\}; \\[4pt]
& I_r^{\,j}   = \langle\, y_1,\, \dots,\, y_i^r,\, \dots,\, y_s,\; x_1,\, \dots,\, x_j,\, \dots,\, x_{n-s},\; y_i x_j \,\rangle,
&& \Gamma_r^{\,j}   = \{1,\, y_i^{r-1},\, y_i^{r-2},\, \dots,\, y_i\}; \\[4pt]
\end{aligned}
\end{equation}

 This yields a total of 
\[
 s(n-s)(r-2)+s + (n-s) = s(n-s)(r-2) + n
\] 
fixed points.

Next, we verify the second assertion for fixed points of type $I_p=I_t^{ij}$, the remaining cases being analogous.
Let $\eta \in \mathrm{Hom}(I_t^{ij}, \mathcal{R}/I_t^{ij})^G$ be an element of the Zariski tangent space to $\mathcal{Y}_{s,n,r}$ at $I_t^{ij}$. 
Since $\mathrm{wt}(y_i x_j) = 0$ and the subspace of weight-zero elements in $\mathcal{R}/I_t^{ij}$ is one-dimensional, spanned by the image of $1$, it follows that $\eta(y_i x_j) \equiv \lambda$ for some $0\neq\lambda \in \mathbb{C}$. 
Next consider $\eta(y_i x_j^t)$. 
On the one hand,
\[
\eta(y_i x_j^t) = x_j^{t-1}\eta(y_i x_j) = \lambda x_j^{t-1}.
\]
On the other hand,
\[
\eta(y_i x_j^t) = y_i \eta(x_j^t) = \mu y_i \cdot y_i^{r-t} = \mu y_i^{r-t+1} \equiv 0 \pmod{I_t^{ij}}.
\]
It follows that $\eta(y_i x_j) = 0$, and therefore
\[
\dim \mathrm{Hom}_{\mathcal{R}}(I_t^{ij}, \mathcal{R}/I_t^{ij})^G \le n.
\]

Note that the $G$-equivariance of $\eta$ implies
\[
\begin{aligned}
&\eta(y_u) = \lambda_u y_i, &&u \in \{1, \dots, \hat{i}, \dots, t\},\\
&\eta(x_v) = \mu_v x_j, &&v \in \{1, \dots, \hat{j}, \dots, n-t\},\\
&\eta(y_i^{r-t}) = \lambda x_j^t,\\
&\eta(x_j^{t+1}) = \mu y_i^{r-t-1},
\end{aligned}
\]
which accounts for a total of $(t-1) + (n-t-1) + 2 = n$ independent parameters. Hence, we conclude that
\(
\dim \mathrm{Hom}_{\mathcal{R}}(I_t^{ij}, \mathcal{R}/I_t^{ij})^G = n.
\)
\end{proof}
 
\begin{thm}\label{smooththm}
The varieties $\mathcal{Y}_{s,n,r}$ are smooth. Moreover, the Hilbert-Chow morphism
\(
\rho: \mathcal{Y}_{s,n,r} \longrightarrow \mathcal{X}_{s,n,r}
\)
is a projective resolution of the quotient singularity $\mathcal{X}_{s,n,r}$.
\end{thm}
\begin{proof}
Consider a generic one-parameter subgroup $\lambda:\mathbb{C}^*\rightarrow T$
with $\lambda(t)\rightarrow 0$ as $t\rightarrow 0$. For any point $I \in \mathcal{Y}_{s,n,r}$, we have that $\lim\limits_{t\rightarrow 0}\lambda(t)\cdot I$ is a $T$-fixed point. As the dimension of Zariski tangent space is an upper semi-continuous function, we conclude that $\mathcal{Y}_{s,n,r}$ is smooth. Thus $\mathcal{Y}_{s,n,r}$ provides a smooth resolution of $\mathcal{X}_{s,n,r}$, and the Hilbert-Chow morphism $\rho$ is projective by construction. 
\end{proof}

\subsection{Toric description of $\mathcal{Y}_{s,n,r}$} Consider the collection of $r-1$ vectors
\begin{equation}\label{eq:Vt}
     v_t := \frac{1}{r} (\underbrace{r-t,\ldots,r-t}_s, \underbrace{t,\ldots,t}_{n-s}),\,1 \le t \le r-1
\end{equation}
in the lattice $N$ (see~\eqref{eq:LatticeN}).  
Let $\widetilde{\mathcal{Y}}_{s,n,r}$ denote the toric variety whose fan is obtained from the cone $\sigma$ (see~\eqref{eq: ConeSigma}) by performing successive star subdivisions along the rays $\mathbb{R}_{\ge 0} v_t$.

To identify the $G$-Hilbert scheme $\mathcal{Y}_{s,n,r}$ with the toric variety $\widetilde{\mathcal{Y}}_{s,n,r}$, we recall the following notion from~\cite{Nak}.

\begin{defn}
A subset $\Gamma$ of monomials in $\mathcal{R}$ is called a \emph{$G$-graph} if
\begin{enumerate}
    \item it contains the constant monomial $1$;
    \item whenever $pq \in \Gamma$, we also have $p \in \Gamma$ and $q \in \Gamma$;
    \item there is a one-to-one correspondence between $\Gamma$ and the irreducible representations of $G$, with respect to the induced action of $G$ on $\mathcal{R}$.
\end{enumerate}
\end{defn}

Using the fact that $G$-graphs are in one-to-one correspondence with torus fixed points, together with our description in~\eqref{eq:fixedPoints}, part~(iii) of Theorem~2.11 in~\cite{Nak} implies the following result.

\begin{thm}\label{ToricHilbThm}
The $G$-Hilbert scheme $\mathcal{Y}_{s,n,r}$ is isomorphic to the toric variety $\widetilde{\mathcal{Y}}_{s,n,r}$.
\end{thm}

\begin{cor}\label{DivisorLem}
Let $\rho$ be the resolution defined in \eqref{eq:Resolution}. 
Then the exceptional locus $\mathrm{Exc}(\rho)$ consists of  $r-1$ irreducible divisors. 
These divisors are in one-to-one correspondence with the rays 
\[
\mathbb{R}_{\ge 0} \, v_t,
\] 
where the vectors $v_t$ are as defined in \eqref{eq:Vt}.
\end{cor}

\begin{rmk}
	The endpoints of the primitive vectors $\{v_1,\ldots, v_{r-1}\}$ lie in a two-dimensional plane defined by the equations
	\[
	\begin{cases}
		z_i = z_j, & 1 \leq i < j \leq s, \\[6pt]
		z_i = z_j, & s+1 \leq i < j \leq n, \\[6pt]
		\frac{1}{s}\sum\limits_{i=1}^s z_i + \frac{1}{n-s}\sum\limits_{j=s+1}^n z_j = 1.
	\end{cases}
	\]
\end{rmk}

\begin{rmk}
    Another way to establish the smoothness of $\mathcal{Y}_{s,n,r}$, as in Corollary~\ref{smooththm}, is to check that the volume of each pyramid $\sigma_t^{ij}$, corresponding to $G$-graph $\Gamma^{ij}$, equals $\frac{1}{r}$. The volume is computed as the absolute value of the determinant of the matrix formed by the primitive vectors along the generating rays of $\sigma_t^{ij}$.
\end{rmk}

We now turn to the discrepancy divisor of the resolution.

\begin{lem}
		The discrepancy divisor of the resolution $\rho_{s,n,r}$ (described in \eqref{eq:Resolution}) is
	\[
	\triangle_{\mathcal{Y}_{s,n,r}}:=K_{\mathcal{Y}_{s,n,r}}-\rho_{s,n,r}^*(K_{\mathcal{X}_{s,n,r}})=\sum_{j=1}^{r-1}a_jE_j,
	\]
	where 
	\(
	a_j=n-s+\frac{j(2s-n)}{r}-1.
	\)
\label{DiscrepDivLemma}
\end{lem}
\begin{proof}
	Since both $\mathcal{X}_{s,n,r}$ and $\mathcal{Y}_{s,n,r}$ are toric varieties, one has 
	\[
	K_{\mathcal{X}_{s,n,r}}=-\sum_{i=1}^n Z_i, \qquad 
	K_{\mathcal{Y}_{s,n,r}}=-\sum_{i=1}^n Z_i-\sum_{j=1}^{r-1} E_j.
	\] 
	 The coefficient of $E_j$ in the pullback $\rho_{s,n,r}^*(K_{\mathcal{X}_{s,n,r}})$ is given by
\[
-\sum_{i=1}^n (Z_i, v_j) = -\frac{js + (n-s)(r-j)}{r} = -\left(n-s + \frac{j(2s-n)}{r} - 1\right),
\]  
where the first equality follows from the formula in Chapter~8 of \cite{CLS}. This computation establishes the claim.
\end{proof}

\begin{cor}\label{discrepDivisor}
	\begin{enumerate}
		\item The resolution $\rho_{s,n,r}$ is crepant (i.e. $\triangle_{\mathcal{Y}_{s,n,r}}=0$) if and only if $s=1$ and $n=2$.
		\item Otherwise, the discrepancy divisor is positive ($a_j>0$ for all $j$).
	\end{enumerate}
\end{cor}

\begin{proof}
For $\triangle_{Y}=0$ one needs $\frac{js+(n-s)(r-j)}{r}-1=0$ for all $j$. This is equivalent to $(2s-n)j+r(n-s-1)=0$ for all $j\in\{1,\ldots,r-1\}$, which gives the system of two linear equations on $r,s$ and $n$: $$\begin{cases}
	2s-n=0\\
	r(n-s-1)=0.
\end{cases}$$ 
Since $r\neq 0$, the only solution is $s=1, n=2$. This completes verification of $(1)$. The verification of statement~(2) proceeds in complete analogy with the proof of Theorem~\ref{singType}.
\end{proof}

\section{Further implications: McKay correspondence}\label{sec4}

In this section, we show that the results of the preceding section imply that the classical McKay correspondence holds for the cyclic groups and resolutions under consideration. Moreover, we provide a natural labeling of the irreducible components of the central fiber by the nontrivial characters of $G$.

\subsection{Classical McKay correspondence}
Let $G\subset GL_n(\mathbb{C})$ be a finite subgroup. Recall that the (classical) McKay correspondence refers to the bijection

\begin{equation}\label{McKayCor}
    \left\{
\begin{aligned}
	&\mbox{nontrivial~irreducible}\\
	&\mbox{representations~of  }G
\end{aligned}\right\}\overset{1:1}{\longleftrightarrow}
\left\{
\begin{aligned}
	&	\mbox{irreducible ~components}\\
	&	\mbox{of ~the ~central~ fiber } \rho^{-1}(0)
\end{aligned}\right\},
\end{equation}
where $\rho: Y \to \mathbb{C}^n // G$ is a resolution of singularities.  In cases where such a bijection exists, we say that the \emph{McKay correspondence holds} for $\rho$ and $G$.

We return to the examples considered in this work. The numerical coincidence observed in Corollary~\ref{DivisorLem}—namely, that the number of exceptional divisors equals the number of nontrivial irreducible representations of the cyclic group $G$—suggests that the classical McKay correspondence holds in this setting.

\begin{thm}\label{McKayCorrThm}
The classical McKay correspondence holds  the resolutions $\rho_{s,n,r}$ described in \eqref{eq:Resolution}.
\end{thm}

\subsection{Derived McKay correspondence}
Notice that $\mathbb{C}[G]$ (the group algebra of $G$) is naturally isomorphic to $K^G(\mathbb{C}^2)$, the Grothendieck group of $G$-equivariant coherent sheaves on $\mathbb{C}^2$. Following this observation, in \cite{GSV}, the McKay correspondence was realized geometrically  as an isomorphism of Grothendieck groups 
\begin{equation}\label{Eq:IsomOfKgroups}
   K^G(\mathbb{C}^2)\rightarrow K(Y). 
\end{equation}

Let $Coh_G(\mathbb{C}^n)$ be the category of $G$-equivariant coherent sheaves on $\mathbb{C}^n$, and $Coh(Y)$
be the category of coherent sheaves on $Y$. 

Any finite-dimensional representation $V$ of $G$ gives rise to two $G$-equivariant sheaves on
$\mathbb{C}^n$: the skyscraper sheaf 
\begin{equation}\label{Eq: SkyscraperSheaf}
V^!=V\otimes_{\mathbb{C}} \mathcal{O}_{0},
\end{equation}
whose fiber at $0$ is $V$ and all the other fibers vanish, and the
locally free sheaf $\widetilde{V}=V\otimes_{\mathbb{C}} \mathcal{O}_{\mathbb{C}^n}$.

In \cite{KV} the isomorphism of Grothendieck groups in Eq. \eqref{Eq:IsomOfKgroups} was lifted to an equivalence of triangulated categories of coherent sheaves: $$D^b(\mathrm{Coh}_G(\mathbb{C}^2))\rightarrow D^b(\mathrm{Coh}(Y)).$$ In particular, under this equivalence, $\chi^!$,
the skyscraper sheaf at $0$ associated to a nontrivial irreducible $G$-representation $\chi$, is mapped to the structure sheaf
of the corresponding exceptional divisor (twisted by $\mathcal{O}(-1)$). Bridgeland, King and Reid constructed the equivalence $D(\mathrm{Coh}_G(\mathbb{C}^3))\rightarrow D(Y)$ for any finite subgroup $G\subset SL_3(\mathbb{C})$ and $Y = G\operatorname{-}\mbox{Hilb}(\mathbb{C}^3)$ (see \cite{BKR}). They showed that $G\operatorname{-}\mbox{Hilb}(\mathbb{C}^3)$  is a crepant resolution of $\mathbb{C}^3$.  

More generally, let $G\subset GL_n(\mathbb{C})$ be a finite subgroup and consider the affine variety $X=\mathbb{C}^n//G:=Spec(\mathbb{C}[z_1,z_2,\ldots,z_n])^G$. Additionally, suppose that

\begin{enumerate}
	\item $X$ has an isolated singularity at $0$;
	\item there exists a projective resolution $\rho: Y \rightarrow X$. 
\end{enumerate}

In modern terms, the McKay correspondence is most commonly formulated as an equivalence of triangulated categories
\[
  \Psi:  D^b(\mathrm{Coh}_G(\mathbb{C}^n)) \;\simeq\; D^b(\mathrm{Coh}(Y)),
\]

which we will refer to as the \emph{derived McKay correspondence}. The equivalence in \eqref{Eq:DerivedMcKay} is known to hold in several important cases:  
\begin{itemize}
	\item $G\subset SL_2{(\mathbb{C})}$, any $G$ (\cite{KV});
	\item $G\subset SL_3{(\mathbb{C})}$, any $G$, $Y=G\operatorname{-}\mbox{Hilb}(\mathbb{C}^3)$ (\cite{BKR});
	\item  $G\subset SP_{2n}(\mathbb{C})$, any $G$ and crepant symplectic resolution  $(Y,\rho)$ (\cite{BK});
	\item $G\subset SL_n{(\mathbb{C})}$, any abelian $G$,  any crepant symplectic resolution $(Y,\rho)$ (\cite{KawT}).
\end{itemize}

\begin{rmk}
    In the present setting with $n \ge 3$, the resolutions $\rho_{s,n,r}$ are  discrepant (see Corollary \ref{discrepDivisor}). 
This discrepancy implies that it is unlikely for the derived McKay correspondence to hold. 
\end{rmk}

\subsection{Images of skyscraper sheaves and marking of exceptional divisors}
The goal of this subsection is to construct a natural bijection appearing in Eq.~\eqref{McKayCor} through the correspondence
\[
\chi \;\longleftrightarrow\; \operatorname{Supp} \, H^0\!\bigl(\Psi(\chi^!)\bigr),
\]
for the finite cyclic subgroups of $GL_n(\mathbb{C})$ considered in \eqref{eq:G-action}. Here, $\chi^!$ and $\Psi$ are defined in Eq.~\eqref{Eq: SkyscraperSheaf} and Eq.~\eqref{Eq:DerivedMcKay}, respectively.
We interpret this relation as a geometric realization of the McKay correspondence in our setting. 

We briefly recall the necessary setup and refer the reader to~\cite{CL} and the references therein for a more comprehensive exposition.  The $G$-Hilbert scheme $\mathcal{Y}_{s,n,r}$ comes equipped with the universal family 
\(
\mathcal{Z} \subset \mathcal{Y}_{s,n,r} \times \mathbb{C}^n
\)
of $G$-clusters, together with the natural projections
\[
p_1 : \mathcal{Y}_{s,n,r} \times \mathbb{C}^n \longrightarrow \mathcal{Y}_{s,n,r}, 
\qquad 
p_2 : \mathcal{Y}_{s,n,r} \times \mathbb{C}^n \longrightarrow \mathbb{C}^n,
\]
illustrated in Figure~\ref{projections}.  

\begin{center}
\begin{figure}[h!]
\centering
\begin{tikzcd}
&& {\mathcal{Y}_{s,n,r} \times \mathbb{C}^n} \\
\\
{\mathcal{Y}_{s,n,r}} &&&& {\mathbb{C}^n}
\arrow["{p_1}"', from=1-3, to=3-1]
\arrow["{p_2}", from=1-3, to=3-5]
\end{tikzcd}
\caption{The product $\mathcal{Y}_{s,n,r} \times \mathbb{C}^n$ together with its two natural projections.}
\label{projections}
\end{figure}
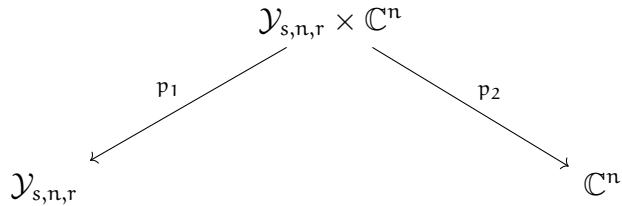
\end{center}

We let $G$ act trivially on $\mathcal{Y}_{s,n,r}$. In this way, it makes sense to consider 
$G$-equivariant coherent sheaves on $\mathcal{Y}_{s,n,r} \times \mathbb{C}^n$. 
Let 
\[
\mathcal{M} \in \mathrm{Coh}_G(\mathcal{Y}_{s,n,r} \times \mathbb{C}^n)
\]
be the structure sheaf of the universal family $\mathcal{Z}\subset \mathcal{Y}_{s,n,r} \times \mathbb{C}^n$. 
By construction, $\mathcal{M}$ is a $G$-equivariant coherent sheaf, flat over $\mathcal{Y}_{s,n,r}$, 
whose fiber over each point of $\mathcal{Y}_{s,n,r}$ is a finite-length $G$-equivariant sheaf on $\mathbb{C}^n$. 

It follows that 
\(
p_{1*}(\mathcal{M}) \in \mathrm{Coh}_G(\mathcal{Y}_{s,n,r})
\)
is a locally free sheaf of finite rank on $\mathcal{Y}_{s,n,r}$. 
Since $G$ acts trivially on $\mathcal{Y}_{s,n,r}$, we may decompose $p_{1*}(\mathcal{M})$ into isotypical components as
\begin{equation}
p_{1*}(\mathcal{M}) \;\cong\; \bigoplus_{\chi \in \mathrm{Irr}(G)} \mathcal{L}_\chi \otimes \chi,    
\end{equation}
where each $\mathcal{L}_\chi$ is a locally free sheaf of rank $\dim(\chi)$. 
These sheaves $\mathcal{L}_\chi$ are usually referred to as \emph{tautological} sheaves, or 
as Gonzalez--Sprinberg--Verdier (GSp-V) sheaves.

Since $\mathbb{C}^n$ is affine, any $G$-equivariant sheaf $\mathcal{F}$ on it is determined by its space of global sections $\Gamma(\mathcal{F})$ 
together with the natural action of $\mathcal{R} \rtimes G$. 
This allows us to define $\widetilde{\mathcal{F}}$, the sheaf obtained from the localization of $\Gamma(\mathcal{F})^*$ with the dual action of $\mathcal{R} \rtimes G$. 
In particular, we let $\widetilde{\mathcal{M}}$ be the fiberwise dual of $\mathcal{M}$.

This data gives rise to a Fourier--Mukai type functor
\begin{equation}\label{Eq:DerivedMcKay}
\Psi \colon D^b(\mathrm{Coh}_G(\mathbb{C}^n)) \;\longrightarrow\; D^b(\mathcal{Y}_{s,n,r}),
\qquad 
\Psi(\bullet) \;:=\; p_{1*}\!\left(\widetilde{\mathcal{M}} \stackrel{L}{\otimes} p_2^*(\bullet)\right).
\end{equation}

The following decomposition of the universal family $\mathcal{M}$ of $G$-clusters appears in \cite{Log} (see Propositions~3.15):
\[
   \mathcal{M} \;=\; \bigoplus_{\chi} \mathcal{L}(-M_{\chi}),
\]
where the sum runs over all irreducible characters $\chi$ of $G$.  
Here, each $G$-Weil divisor $M_{\chi}$ can be expressed as
\[
   M_{\chi} \;=\; \sum_{i=1}^{r-1} m_{\chi,i} E_i,
\]
where the coefficients $m_{\chi,i}$ are determined by the following rule:
\[
   m_{\chi,t} \;=\; \inf \Bigl\{\, v_t\cdot m \;\Bigm|\; 
   m=(m_1,\ldots,m_n) \in \mathbb{Z}_{\geq 0}^n,\;
   \mathbb{C}\langle z_1^{m_1}\cdots z_n^{m_n}\rangle \simeq \chi 
   \;\text{as $G$-representations}\Bigr\}.
\]

In particular, one obtains
\begin{equation}\label{eq:Mdivisors}
\begin{aligned}
&M_{{\chi_0}} = 0, \\[4pt]
&M_{{\chi_1}} = \tfrac{1}{r}\sum_{t=1}^{r-1} t\,E_t, \\[4pt]
&M_{{\chi_{2}}} = \tfrac{1}{r}\sum_{t=1}^{r-2} 2t\,E_t + \tfrac{r-2}{r} E_{r-1}, \\[4pt]
&M_{{\chi_{3}}} = \tfrac{1}{r}\sum_{t=1}^{r-3} 3t\,E_t + \tfrac{2(r-3)}{r} E_{r-2} + \tfrac{r-3}{r} E_{r-1}, \\[4pt]
&\hspace{3cm} \vdots \\[4pt]
&M_{{\chi_{r-1}}} = \tfrac{1}{r}\sum_{t=1}^{r-1} (r-t) E_t .
\end{aligned}
\end{equation}

\begin{ex}
    Let $G=\mathbb{Z}/5\mathbb{Z}$ and consider the embedding 
    \[
      \varphi: G \hookrightarrow GL_6(\mathbb{C}), 
      \quad \varphi(1) = \mathrm{diag}(\varpi,\varpi,\varpi^{-1},\varpi^{-1},\varpi^{-1},\varpi^{-1}),
    \]
    where $\varpi$ is a primitive $5$-th root of unity.  
    Here $s=2$, and the coefficients of the corresponding $G$-Weil divisors in terms of the exceptional divisors on $\mathcal{Y}_{2,6,5}$ are listed in Table~\ref{DivDecompTable}.  

    For example, the coefficient $m_{\chi_3,3}$ of $E_3$ in the decomposition of $M_{\chi_{3}}$ is computed as follows.  
The character $\chi_3$ corresponds to monomials on which $G$ acts by multiplication with $\varpi^3$.  
Among these monomials, the ones whose exponent vector have the smallest dot product with 
\[
v_3 = \tfrac{1}{5}(2,2,3,3,3,3)
\]
is $z_1 z_2$, since $G$ acts on it by multiplication with $\varpi^{-1}\cdot\varpi^{-1}=\varpi^{-2}=\varpi^3$ as required.  
Therefore,
\[
m_{\chi_3,3} = v_3\cdot (1,1,0,0,0,0) = \frac{2+2}{5} = \frac{4}{5}.
\]

    \begin{table}[h!]  
	\centering
	\begin{tabular}{|c|c|c|c|c|}
		\hline
		    &  $E_1$ & $E_2$  & $E_3$  & $E_4$   \\
		\hline
		$M_{{\chi_{0}}}$ & $0$ &  $0$ & $0$  & $0$  \\
		\hline
		$M_{{\chi_{1}}}$ & $1/5$ &  $2/5$ & $3/5$  & $1/5$ \\
		\hline
        $M_{{\chi_{2}}}$ & $2/5$ &  $4/5$ & $6/5$  & $3/5$ \\
		\hline
        $M_{{\chi_{3}}}$ & $3/5$ &  $6/5$ & $4/5$  & $2/5$  \\
		\hline
        $M_{{\chi_{4}}}$ & $4/5$ &  $3/5$ & $2/5$  & $1/5$ \\
		\hline
	\end{tabular}
        \caption{Coefficients of $G$-Weil divisors $M_{\chi_k}$ in terms of exceptional divisors $E_\ell$ for $G=\mathbb{Z}/5\mathbb{Z}$.}
        \label{DivDecompTable}
    \end{table}
\end{ex}

\begin{rmk}
    The coefficients satisfy certain symmetry relations; for example, one checks that
    \[
       m_{\chi_k,\ell} \;=\; m_{\chi_{r-\ell},\,r-k}, \qquad k>0.
    \]
  
\end{rmk}

According to the formula in Proposition $3.2$ in \cite{Log1}, the principal divisors of the coordinate functions are given by
\begin{equation}\label{eq:DivOfCOoordFunc}
    \begin{aligned}
        & (z_i) 
\;=\; \mathcal{Z}_i \;+\; \frac{1}{r}\sum_{t=1}^{r-1} t E_t,
&& 1 \leq i \leq s;\\
& (z_j) 
\;=\; \mathcal{Z}_{s+j} \;+\; \frac{1}{r}\sum_{t=1}^{r-1} (r-t)E_t,&& s+1 \leq j \leq n.
    \end{aligned}
\end{equation}
The image of $\chi^!_{t}$ for finite abelian subgroups 
$G \subset GL_n(\mathbb{C})$ is described in Corollary~$2.5$ of \cite{CL}. 
Using this result, we obtain that 
\(
\Psi(\chi^!_{t})
\)
is represented by the following complex:
\begin{equation}
\mathcal{L}_{n-2s-t} 
\;\longrightarrow\; 
\mathcal{L}_{n-2s-t-1}^{\oplus s} \oplus \mathcal{L}_{n-2s-t+1}^{\oplus (n-s)} 
\;\longrightarrow\; \cdots \;\longrightarrow\;
\mathcal{L}_{-t+1}^{\oplus s} \oplus \mathcal{L}_{-t-1}^{\oplus (n-s)}
\;\longrightarrow\; \mathcal{L}_{-t},
\label{eq:image-of-chi}
\end{equation}
where $\mathcal{L}_{a} \in \mathrm{Coh}(\mathcal{Y}_{s,n,r})$ denotes the line bundle, on which $G$ acts by the character $\chi_a$. The maps in the complex  are induced by multiplication with the coordinate functions:
\begin{equation}
\begin{aligned}\label{eq:maps}
    &\mathcal{L}_{a} \xrightarrow{\cdot z_i} \mathcal{L}_{a-1}, 
    \qquad 1 \leq i \leq s,\\
    &\mathcal{L}_{a} \xrightarrow{\cdot z_j} \mathcal{L}_{a+1}, 
    \qquad s+1 \leq j \leq n.
\end{aligned}
\end{equation}

For each character $\chi_a$ and coordinate function $z_i$, let
\(\mathcal{B}_{\chi_a, z_i}\) denote the effective Weil divisor on $\mathcal{Y}_{s,n,r}$ 
where the morphism 
\(\mathcal{L}_{a} \xrightarrow{\cdot z_i} \mathcal{L}_{a\pm 1}\) vanishes. 
These divisors are exactly the vanishing loci of the maps in \eqref{eq:image-of-chi}. 
Using the explicit formulas in \eqref{eq:Mdivisors} and \eqref{eq:DivOfCOoordFunc}, they can be computed as
follows:
\begin{equation}\label{eqVanishingDivisors}
\begin{aligned}
    &\mathcal{B}_{\chi_a, z_i} 
= M_{\chi_a \otimes \chi_{r-1}} + (z_i) - M_{\chi_a}
= \mathcal{Z}_i + \sum_{t=r-a+1}^{r-1} E_t,
&& 1 \leq i \leq s,\\
&\mathcal{B}_{\chi_a, z_j} 
= M_{\chi_a \otimes \chi_{1}} + (z_j) - M_{\chi_a}
= \mathcal{Z}_{j} + \sum_{t=1}^{r-a-1} E_t,
&& s+1 \leq j \leq n,
\end{aligned}
\end{equation}
(see Section~4.6 of \cite{Log2}).

The top and bottom cohomologies of the images of skyscraper sheaves corresponding to irreducible representations 
of $G$ in $D(\mathcal{Y}_{s,n,r})$ are described as follows.

\begin{prop}
Let $\chi_t$ be an irreducible character of $G$. Then:
\begin{enumerate}
    \item  \(
        H^0\big(\Psi(\chi^!_t)\big) 
        \;\simeq\; \mathcal{L}_{-t} \otimes \mathcal{O}_{E_t}.
    \)
    \item \(
        H^{-n}\big(\Psi(\chi^!_t)\big) = 0 \quad \text{for } t \neq n-2s\), 
   \item 
      \( H^{-n}\big(\Psi(\chi^!_{n-2s})\big) \;\simeq\; \mathcal{O}_{E_1 \cup \dots \cup E_{r-1}}.
    \)
\end{enumerate}
\label{prop:top-bottom-cohomology}
\end{prop}

\begin{proof}
The zeroth cohomology $H^0(\Psi(\chi^!_t))$ is realized as the cokernel of the map
\[
\mathcal{L}_{r-t+1}^{\oplus s} \oplus \mathcal{L}_{r-t-1}^{\oplus (n-s)} \;\longrightarrow\; \mathcal{L}_{r-t}.
\]
By definition, this cokernel is
\[
\mathcal{L}_{r-t} \otimes \mathcal{O}_Z,
\]
where $Z$ is the intersection of the vanishing loci of the maps in \eqref{eq:maps}:
\[
Z = \bigcap_{i=1}^s \mathcal{B}_{\chi_{r-t+1}, z_i} \;\cap\; \bigcap_{j=s+1}^n \mathcal{B}_{\chi_{r-t-1}, z_j}.
\]

Using the explicit formulas for $\mathcal{B}_{\chi_a, z_i}$ from \eqref{eqVanishingDivisors}, for $1\le i\le s$ and $s+1\le j\le n$ we compute
\[
\mathcal{B}_{\chi_{r-t+1}, z_i} = \mathcal{Z}_i + \sum_{i=t}^{r-1} E_i, 
\qquad 
\mathcal{B}_{\chi_{r-t-1}, z_{j}} = \mathcal{Z}_{j} + \sum_{i=1}^{t} E_i,
\]
which gives
\(
Z = E_t,
\)
completing verification of the first statement.

Similarly, the kernel of the map
\[
\mathcal{L}_{n-2s+r-t} \;\longrightarrow\; 
\mathcal{L}_{n-2s+r-t-1}^{\oplus s} \oplus \mathcal{L}_{n-2s+r-t+1}^{\oplus (n-s)}
\]
is supported on the intersection
\[
\bigcap_{i=1}^n \mathcal{B}_{\chi_{n-2s+r-t}, z_i}.
\]
It follows from the explicit formulas in \eqref{eqVanishingDivisors} that this intersection is empty whenever 
\(
n - 2s + r - t \not\equiv 0 \pmod{r},
\) 
and, when the congruence holds, it coincides with the entire exceptional locus of \(\rho_{s,n,r}\), which is the union of \(r-1\) irreducible components.

Hence, the stated formulas for the top cohomology follow.
\end{proof}

\section{Realization of $\mathcal{Y}_{s,n,r}$ as a quiver variety and Implications}\label{sec5}

We now present an alternative description of $\mathcal{Y}_{s,n,r}$, due to Y.~Ito and H.~Nakajima (see Sections~3 and~4 of \cite{IN}). 
Let us briefly recall the setup. Let $G \subset GL_n(\mathbb{C})$ be a finite subgroup, and let 
\[
Q = (G,\mathbb{C}^n)
\]
denote the McKay quiver associated to the action of $G$. 
The set of vertices of $Q$ is in bijection with the irreducible representations of $G$. 
For vertices corresponding to irreducible representations $\chi_k$ and $\chi_s$, the number of edges between them is given by 
\(
\dim \operatorname{Hom}_G(\chi_k,\, \chi_s \otimes \mathbb{C}^n).
\)

Fix the dimension vector \( v = (1,1,\ldots,1) \); that is, assign to each vertex \( i \) of \( Q \) a one-dimensional vector space \( V_i \).  
For each arrow \( i \to j \) in \( Q \), attach a linear map between the corresponding spaces, which is an element of \( \mathrm{Hom}(V_i, V_j) \).  
Since \( \dim(V_i) = \dim(V_j) = 1 \), each such map is determined by a single scalar, i.e.\ \( \mathrm{Hom}(V_i, V_j) \simeq \mathbb{C} \).

For our choice of the pair \( (G, Q) \), the only nonzero maps occur between \( V_i \) and \( V_j \) whenever \( |i - j| \equiv 1 \pmod r \).  
We denote by \( y \) the maps that increase the index and by \( x \) those that decrease it (by one in both cases).
A collection of scalars
\[
(x^{i+1}_{\ell},\, y^{i}_{t}) 
\quad \text{with } t, \ell \in \{1,\ldots,n\}, \; i \in \{0,\ldots,r-1\},
\]
subject to the commutativity relations
\begin{equation}
\begin{split}
x^{i+1}_{\ell}y^{i}_{t} &= y^{i+1}_{t}x^{i+2}_{\ell}, \\
y^{i}_{t}y^{i+1}_{\ell} &= y^{i}_{\ell}y^{i+1}_{t}, \\
x^{i+1}_{\ell}x^{i+2}_{t} &= x^{i+1}_{t}x^{i+2}_{\ell},
\end{split}
\label{McKayEqns}
\end{equation}
defines a representation of \( Q \) with dimension vector \( v \).

A schematic depiction of such a representation is shown in Figure~\ref{McKayQuiver}.  
For notational convenience, we group the maps as
\[
Y_{i}:=\{y_{1}^1,y_{1}^2,\ldots,y_{i}^s\}, 
\quad 
X_{j}:=\{x_{j}^1,x_{j}^2,\ldots,x_{j}^{\,n-s}\}.
\]

\begin{center}
	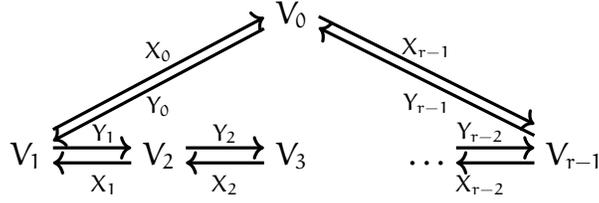
\begin{figure}[h!]
		\begin{tikzpicture}
		\matrix(m)[matrix of math nodes,
		row sep=3em, column sep=2.5em,
		text height=1.5ex, text depth=0.25ex]
		{& & V_0\\ V_1 & V_2 &V_3 & \ldots  & V_{r-1}\\};
		\path[->,font=\scriptsize] 
		(m-1-3) edge[very thick] node[below] {$Y_{0}$} (m-2-1)
		(m-2-1)  edge[very thick, transform canvas={yshift=1.5mm}] node[above] {$X_{0}$} (m-1-3)  
		(m-1-3)  edge[very thick, transform canvas={yshift=1.5mm}] node[above] {$X_{r-1}$} (m-2-5)
		(m-2-5) edge[very thick] node[below] {$Y_{r-1}$} (m-1-3) 
		(m-2-1)  edge[very thick, transform canvas={yshift=1mm}] node[above,yshift=-1mm, xshift=1.6mm] {$Y_{1}$} (m-2-2)
		(m-2-2)  edge[very thick, transform canvas={yshift=-1mm}] node[below, xshift=1.6mm] {$X_{1}$}  (m-2-1) 
		(m-2-2)  edge[very thick, transform canvas={yshift=1mm}] node[above,yshift=-1mm] {$Y_{2}$} (m-2-3)
		(m-2-5)  edge[very thick, transform canvas={yshift=-1mm}] node[below, xshift=-2mm] {$X_{r-2}$} (m-2-4)
		(m-2-4)  edge[very thick, transform canvas={yshift=1mm}] node[above,yshift=-1mm,  xshift=-2mm] {$Y_{r-2}$} (m-2-5)
		(m-2-3)  edge[very thick, transform canvas={yshift=-1mm}] node[below] {$X_{2}$} (m-2-2);
		\end{tikzpicture}
		\caption{McKay quiver $Q=(G,\mathbb{C}^n)$}
		\label{McKayQuiver}
	\end{figure}
\end{center}

Let $\mathcal{I}\subset \mathbb{C}[y_{i}^{t},x_{j}^{\ell}]$ be the ideal generated by the relations \eqref{McKayEqns}, and set 
\[
\mathcal{R}':=\mathbb{C}[y_{i}^{t},x_{j}^{\ell}]/\mathcal{I}.
\]  
Then the affine scheme of representations of $Q$ is 
\(
\mathcal{M}_Q(v):=\operatorname{Spec}(\mathcal{R}').
\)

Next, we consider the moduli space of such representations. Define  
\begin{equation}\label{eq:defOfG}
    \mathcal{G}:=PGL_G(R)=P\left(\prod_{k=0}^{r-1}GL_G(\chi_k)\right)=P((\mathbb{C}^*)^r)\simeq (\mathbb{C}^*)^{r-1}
\end{equation}

to be the group of $G$-equivariant automorphisms of $R$, the space of representations of the McKay quiver $Q$ with  dimension vector $v=(1,1,\ldots,1)$. Let $\theta: \mathcal{G}\to \mathbb{C}^*$ be the character given by  
\[
\theta(t_1,\ldots,t_{r-1})=t_1^{-1}\cdots t_{r-1}^{-1},
\]  
and form the categorical quotient  
\[
\mathcal{M}_Q^{\theta}(v):=\mathcal{Z}//_{\theta}\mathcal{G}
=\operatorname{Proj}\!\left(\bigoplus_{i>0}\mathcal{R}^{\theta,i}\right)
=\mathcal{Z}^{\theta\text{-ss}}/\!\sim_{\mathcal{G}},
\]  
where  
\[
\mathcal{R}^{\theta,i}=\{f\in \mathcal{R}'\mid f(g^{-1}z)=\theta(g)^n f(z), \ \forall g\in \mathcal{G},\ z\in \mathcal{Z}\}.
\]  
Here, two points $x,y\in \mathcal{Z}^{\theta\text{-ss}}$ are equivalent ($x\sim_{\mathcal{G}}y$) if the closures of their $\mathcal{G}$-orbits intersect.  

Notice that for our choice of $\theta$, a point 
\[
z = (\mathbf{Y}_i, \mathbf{X}_j) \in \mathcal{Z}^{\theta\text{-ss}}
\]
is $\theta$-semistable if and only if every nonzero vector $v_0 \in V_0$ is \textit{cyclic}; that is, the vectors obtained by applying all compositions of the maps $y_i^{t}$ and $x_j^{\ell}$ to $v_0$ span each of the spaces $V_i$. 

Since the $G$-Hilbert schemes $\mathcal{Y}_{s,n,r}$ are smooth (Theorem~\ref{smooththm}), we can apply Corollary~4.6 of~\cite{IN} to obtain the following identification:

\begin{thm}\label{QuivVarThm}
There exists an isomorphism of varieties
\[
    \mathcal{Y}_{s,n,r} \;\simeq\; \mathcal{M}_Q^{\theta}(v),
\]
identifying the $G$-Hilbert scheme with the moduli space of $\theta$-stable representations of the quiver $Q$ with dimension vector $v$.
\end{thm}

 \subsection{Geometry and properties of exceptional divisors}

The exceptional divisors admit the following quiver-theoretic description.

Each exceptional divisor $E_i$ (one of the $r-1$ divisors described in Corollary~\ref{DivisorLem}) corresponds to those representations in $\mathcal{M}_Q^{\theta}(v)$ whose nonzero maps are supported on an acyclic subquiver of $Q$ in which the vertex $i$ is a sink.  
Figure~\ref{DivisorQuiver} illustrates such a subquiver, along with an example of the intersection of two divisors.  
For comparison, Figure~\ref{FixedPtQuiver} depicts a representation corresponding to one of the torus \( T \)-fixed point (see Proposition~\ref{Smooth}).  
In both figures, only the explicitly drawn and labeled arrows may carry nonzero maps; all other maps are understood to be zero.

\begin{center}
\begin{figure}[h!]
	\begin{tikzpicture}
	\matrix(m)[matrix of math nodes,
	row sep=3em, column sep=2.5em,
	text height=1.5ex, text depth=0.25ex]
	{& & V_0\\ V_1 & V_2 &V_3 & \ldots  & V_{r-1}\\};
	\path[->,font=\scriptsize] 
	(m-1-3) edge[very thick] node[below,yshift=2.5mm, xshift=5mm] {$Y_{0}$} (m-2-1)  
	(m-1-3)  edge[very thick] node[above,yshift=-3.5mm, xshift=8mm] {$X_{r-1}$} (m-2-5)
		(m-2-5)  edge[very thick] node[below, xshift=2mm] {$X_{r-2}$} (m-2-4)
	(m-2-1)  edge[very thick] node[above,yshift=-1mm, xshift=1.5mm] {$Y_{1}$} (m-2-2) 
	(m-2-3)  edge[very thick] node[below] {$X_{2}$} (m-2-2)
    (m-2-4)  edge[very thick] node[below] {$X_{3}$} (m-2-3);
	\end{tikzpicture}\hspace{0.05in}	\begin{tikzpicture}
		\matrix(m)[matrix of math nodes,
		row sep=3em, column sep=2.5em,
		text height=1.5ex, text depth=0.25ex]
		{& & V_0\\ V_1 & V_2 &V_3 & \ldots  & V_{r-1}\\};
		\path[->,font=\scriptsize] 
		(m-1-3) edge[very thick] node[below,yshift=2.5mm, xshift=5mm] {$Y_{0}$} (m-2-1)  
		(m-1-3)  edge[very thick] node[above,yshift=-3.5mm, xshift=8mm] {$X_{r-1}$} (m-2-5)
		(m-2-5)  edge[very thick] node[below, xshift=2mm] {$X_{r-2}$} (m-2-4)
		(m-2-1)  edge[very thick] node[above,yshift=-1mm, xshift=1.5mm] {$Y_{1}$} (m-2-2) 
		(m-2-4)  edge[very thick] node[below] {$X_{3}$} (m-2-3);
		\end{tikzpicture}
	\caption{Divisor $E_2$ and intersection of two divisors $E_2\cap E_3$}
	\label{DivisorQuiver}
\end{figure}
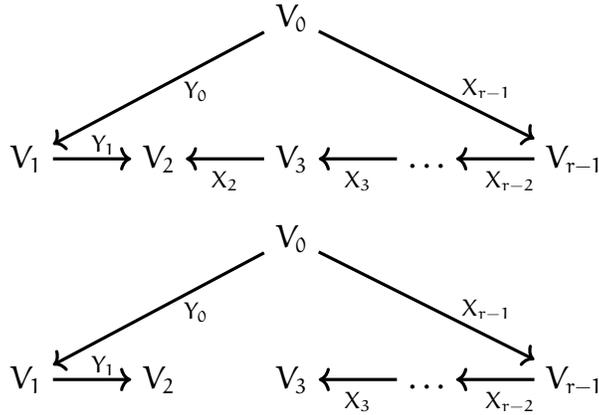
\end{center}

\begin{center}
	\begin{figure}[h!]
		\begin{tikzpicture}
		\matrix(m)[matrix of math nodes,
		row sep=3em, column sep=2.5em,
		text height=1.5ex, text depth=0.25ex]
		{& & V_0\\ V_1 & V_2 &V_3 & \ldots  & V_{r-1}\\};
		\path[->,font=\scriptsize] 
		(m-1-3) edge[very thick] node[below,yshift=2.5mm, xshift=3mm] {$y_{0}^i$} (m-2-1)  
		(m-1-3)  edge[very thick] node[above,yshift=-3.5mm, xshift=8mm] {$x_{r-1}^j$} (m-2-5)
		(m-2-5)  edge[very thick] node[below, xshift=2mm] {$x_{r-2}^j$} (m-2-4)
		(m-2-1)  edge[very thick] node[above,yshift=-1mm, xshift=1.5mm] {$y_{1}^i$} (m-2-2) 
		(m-2-4)  edge[very thick] node[below] {$x_{3}^j$} (m-2-3);
		\end{tikzpicture}
		\caption{Fixed point $I^{ij}_{r-2}$}
		\label{FixedPtQuiver}
	\end{figure}
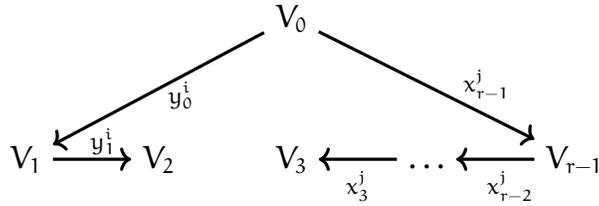
\end{center}


 The following corollary records the pairwise (non)emptiness of intersections among the exceptional divisors \( \{E_i\} \)
and the concrete identification of the extremal divisors in two special cases.
Its verification follows directly from the description of \( \mathcal{Y}_{s,n,r} \) in Theorem~\ref{QuivVarThm}.

\begin{cor}
 The exceptional divisors satisfy 
  \begin{enumerate}
    \item  \(
    E_i \cap E_j = \varnothing \quad \text{if } |i-j| > 1, 
    \qquad 
    E_i \cap E_j \neq \varnothing \quad \text{if } |i-j| = 1;
    \)
    \item if \( s = r-1 \), then \( E_{r-1} \simeq \mathbb{CP}^{r-1} \),
    if \( s = 1 \), then \( E_1 \simeq \mathbb{CP}^{r-1} \).
\end{enumerate}
\end{cor}

\begin{rmk}
It is instructive to compare this result with the classical situation.  For $n=2$ and $s=1$, each exceptional divisor is isomorphic to the projective line $\mathbb{CP}^1$ (see \S 12.4 in \cite{Kir}).  
In contrast, for $n=3$ one finds that $E_1 \simeq \mathbb{CP}^2$, while the remaining $r-2$ exceptional divisors are isomorphic to Hirzebruch surfaces (see Lemma~2.2 in \cite{KED}). 
\end{rmk}

\subsection{Connectedness and irreducibility of $\mathcal{Y}_{s,n,r}$}

We begin by describing an explicit affine covering of $\mathcal{M}_Q^{\theta}(v)$ and, via the isomorphism in Theorem \ref{QuivVarThm}, of the Hilbert scheme $\mathcal{Y}_{s,n,r}$.

The affine charts $\{ U_{p} \}_{p \in \mathcal{Y}_{s,n,r}^T}$ are naturally indexed by the $T$-fixed points 
$p \in \mathcal{Y}_{s,n,r}^T$, which are described explicitly in the proof of Proposition~\ref{Smooth}. 
For a given $T$-fixed point $p = I^{ij}_{t}$, the corresponding affine chart $U_{p}$ is defined by the nonvanishing conditions
\begin{equation}\label{eq:AffCoord}
    y_{k}^i \neq 0 \quad \text{for } 0 \le k < t, 
    \qquad 
    x_{k}^j \neq 0 \quad \text{for } t < k \le r-1.
\end{equation}

On this chart, using the $\mathcal{G}$-action (see \eqref{eq:defOfG}), we can normalize the coordinates in \eqref{eq:AffCoord} to $1$, and then choose the remaining coordinates in agreement with the defining equations in \eqref{McKayEqns}.

The collection of such charts
\(
   \{ U_{p} \}_{p \in \mathcal{Y}_{s,n,r}^T}
\)
forms an open affine covering of $\mathcal{Y}_{s,n,r}$.

We now proceed to establish the following geometric properties of $\mathcal{Y}_{s,n,r}$.

\begin{lem} The scheme $\mathcal{Y}_{s,n,r}$ is connected and  irreducible.
\label{lem:connected_irred}
\end{lem}

\begin{proof}
The connectedness follows since the affine charts cover $\mathcal{Y}_{s,n,r}$ and their common intersection is nonempty (contains the point with all coordinates equal to $1$).  
As $\mathcal{Y}_{s,n,r}$ is smooth (Corollary~\ref{smooththm}) and connected, it is also irreducible.
\end{proof}

\end{document}